\documentclass[12pt]{article}
\usepackage{amsmath, amssymb, amsthm, graphicx, tikz}
\usepackage[margin=1in]{geometry}
\usepackage{enumitem}
\usepackage{forest}
\usepackage{tikz}
\usetikzlibrary{trees}
\usepackage{setspace}
\usepackage{xcolor}
\usepackage{tcolorbox}
\usepackage{float}
\usetikzlibrary{backgrounds}   
\usetikzlibrary{positioning}    
\usetikzlibrary{fit}            
\usepackage[utf8]{inputenc}
\usepackage{graphicx}   
\usepackage{caption}    
\usepackage{subcaption} 
 
\newtheorem{observation}{Observation}
\newtheorem{theorem}{Theorem}
\newtheorem{definition}{Definition}
\newtheorem{lemma}{Lemma}
\newtheorem*{conjecture}{Conjecture}

\begin{document}

\title{On the analogue of Esperet’s conjecture: Characterizing hereditary classes}
\author{
N. Rahimi$^1$\thanks{ORCID: 0009-0004-9571-669X}\ , D.A. Mojdeh$^2$\thanks{Corresponding author, ORCID: 0000-0001-9373-3390}\\
$^{1,2}$Department of Mathematics, Faculty of Mathematical Sciences,\\ University of Mazandaran, Babolsar, Iran\\
 $^{1}${\tt narjesrahimi1365@gmail.com}\\
 $^{2}${\tt damojdeh@umz.ac.ir}  
}
\date{}
\maketitle

\begin{abstract}
In the paper [J. Graph Theory (2023) 102:458-471, the Esperet’s conjecture has been posed: Every \(\chi\)-bounded hereditary class is poly-\(\chi\)-bounded]. 
This conjecture was first posed in [Habilitation Thesis, Université Grenoble Alpes, 24, 2017]. This is adapted from the Gyárfás–Sumner's conjecture which has been asserted in [The Theory and Applications of Graphs, (G. Chartrand, ed.), John Wiley \& Sons, New York, 1981, pp. 557-576].

Although the Esperet’s conjecture is false in general,  but in this study, we consider an analogue of Esperet’s conjecture as follows: Let \(\mathcal{C}\) be a hereditary class of graphs, and \(d \geq 1\). Suppose that there is a function \(f\) such that $\chi(G) \leq f(\tau_d(G))$ for each \(G \in \mathcal{C}\). Can we always choose \(f\) to be a polynomial? We investigate this conjecture by focusing on specific classes of graphs. This work identifies hereditary graph classes that do not contain specific induced subdivisions of claws and confirms that they adhere to the stated conjecture.
\end{abstract}

{\bf Keywords}: Coloring, Esperet’s conjecture, chromatic number, poly-\(\chi\)-bounded.

{\bf 2020 Mathematical Subject classification}: 05C15
\section{Introduction}
The chromatic number is a central concept in graph theory with applications in both mathematics and computer science. A key question is whether large chromatic number forces the presence of specific local substructures. A $k$-coloring of a graph is function
from its vertices to $\{1, \ldots, k\}$
 so that adjacent vertices have different colors. The notations $\chi(G)$ and $\omega(G)$ denote the chromatic number of $G$ and the size of the maximum clique of $G$, respectively. The chromatic number is the smallest integer $k$ such that $G$ has a $k$-coloring.
Given that the \emph{clique number} of a graph is the size of its largest complete subgraph, it is clear that for every graph $G$, we have $\chi(G) \geq \omega(G)$, since all vertices in a clique must receive distinct colors.
But when does the chromatic number exceed the clique number?
Does a high chromatic number necessarily imply the existence of a large clique? The answer is no. In the 1940s, Tutte \cite{Descartes,Scott} showed that there exist triangle-free graphs with arbitrarily large chromatic number. Later, Mycielski \cite{Myciel} provided explicit constructions of such graphs.
This raises another question, can graphs with high chromatic number resemble trees? Erdős \cite{ERDO} answered this question in the 1950s by showing the existence of graphs with both large chromatic number and arbitrarily large girth.\\

A graph class \( C \) is called \emph{hereditary} if it is closed under taking induced subgraphs. It is said to be \(\chi\)-bounded if there exists a function \( f \) such that \(\chi(G) \leq f(\omega(G))\) for every graph \( G \in C \). If such a function \( f \) can be chosen to be a polynomial, then \( C \) is called \emph{poly-\(\chi\)-bounded} \cite{GYAR2, Scott-sem1}. Although many classes are known to be \(\chi\)-bounded, their proofs typically produce bounding functions that grow quickly. A graph is $H$-free if it contains no induced copy of $H$.

 \begin{conjecture}\emph{(Gyárfás-Sumner \cite{GYAR COV,Sumner})}\label{Gyárfás-Sumner}
 For every forest $H$, and every positive integer $k$, every $H$-free graph that does not contain a clique on $k$ vertices has bounded chromatic number.
 \end{conjecture}
 
 This is known only for a few special kinds of forest \cite{Scott-sem1}. During the 1980s, Erdős and Hajnal proposed their well-known conjecture (see \cite{EH1,EH2}). 
  
  \begin{conjecture}\emph{(Erdős-Hajnal \cite{EH1,EH2})}\label{Erdős-Hajnal}
 For every graph $H$, there is a constant $c=c(H)>0$ such that the following
holds: every $H$-free graph with $n$ vertices has a complete subgraph or stable set with at
least $n^c$ vertices.
 \end{conjecture}

A hereditary class \( \mathcal{C} \) of graphs has the Erdős–Hajnal property if there is \( c > 0 \) such that every \( G \in \mathcal{C} \) has a stable set or complete subgraph with at least \( |G|^c \) vertices. Thus the Erdős–Hajnal Conjecture says that the class of \( H \)-free graphs has the Erdős–Hajnal property\cite{Scott}.

\begin{theorem}\emph{(\cite{Chudnovs1})}\label{C5}
The Erdős–Hajnal Conjecture holds when $H$ is a cycle of length $5$.
\end{theorem}

Polynomially $\chi$-bounded classes are of particularly interest because of their connection to the Erdős–Hajnal Conjecture: it follows immediately that any polynomially $\chi$-bounded class has the Erdős–Hajnal property \cite{Scott}.

\begin{theorem}\emph{(\cite{Scottsysp2})}\label{P5}
 Every graph with chromatic number at least $k^{\log_2 k}$ contains either a clique on $k$ vertices or an induced path on five vertices.
 \end{theorem}

Although the bound is only slightly superpolynomial, it is still insufficient to prove the Erdős--Hajnal Conjecture for $P_5$.

\begin{conjecture}\emph{(Esperet \cite{Esperet2017})}
 Every \(\chi\)-bounded hereditary class is poly-\(\chi\)-bounded.
\end{conjecture}

However, this was recently disproved \cite{Briański}, the main open question now is to identify which hereditary classes are poly-$\chi$-bounded.

For a graph \( G \) and an integer \( d \geq 1 \), let \(\tau_d(G)\) denote the largest \( t \) such that \( G \) has a subgraph (not necessarily induced) isomorphic to the complete \( d \)-partite graph with each part of cardinality \( t \).

While Esperet’s Conjecture remains open, it may be preferable to replace a clique  with a complete $d$-partite graph. This question has been posed in \cite{Scott-sem2}.

(An analogue of Esperet’s (false) conjecture for \(\tau_d(G)\)): Let \(\mathcal{C}\) be a hereditary class of graphs, and \(d \geq 1\). Suppose that there is a function \(f\) such that
\[
\chi(G) \leq f(\tau_d(G))
\]
for each \(G \in \mathcal{C}\). Can we always choose \(f\) to be a polynomial?

The results obtained so far in this field.

\begin{theorem}\emph{(\cite{Scott-sem2} 1.6)}\label{theodoe24}
For every tree \( H \) of radius two, and every integer \( d \geq 1 \), there is a polynomial \( f \) such that
\[
\chi(G) \leq f(\tau_d(G))
\]
for every \( H \)-free graph \( G \).
\end{theorem}

\begin{theorem}\emph{(\cite{Nguyen} 2.4)}\label{theotu23-1}
For every path $H$ and all $d \ge 1$, there is a polynomial $f$ such that, for all $t \ge 1$,  
if a graph $G$ is $H$-free and does not contain $K_d(t)$ as a subgraph, then $\chi(G) \le f(t)$.
\end{theorem}

\begin{theorem}\emph{(\cite{Nguyen} 3.2)}\label{theotu23-2}
For all $d \ge 1$ and every broom $H$ 
(A broom is obtained from a path with one end v by adding leaves adjacent to v), there is a polynomial $f$ such that, for all $t \ge 1$,  
if a graph $G$ is $H$-free and does not contain $K_d(t)$ as a subgraph, then $\chi(G) \le f(t)$.
\end{theorem}

A \emph{rooted tree} $(T, r)$ consists of a tree $T$ and a vertex $r$ of $T$ called the root. The \emph{height} of $(T, r)$ is the length (number of edges) of the longest path in $T$ with one endpoint $r$. The \emph{spread} of $T$ is the maximum over all vertices $u \in V(T)$ of the number of children of $u$.

For integers $s,p\geq 1$, let $H(s,p)$ be the tree with $1 + s + s^2 + \dots + s^p$ vertices, where some vertex has degree $s$ (root) and all its neighbors have degree $s+1$ and the same process should continue for the vertices at each level until reaching level $p$, so $H(s,p)$ has height $\eta = p$ (at the first level, the height is zero , $p=0$) and spread $\zeta = s$.

\vspace{2\baselineskip}
If \( X \subseteq V(G) \), we define \( \chi(X) = \chi(G[X]) \). Let \( X, Y \) be disjoint subsets of \( V(G) \), where \( |X| = t \) and every vertex in \( Y \) is adjacent to every vertex in \( X \). We call \( (X, Y) \) a \textbf{\( t \)-biclique}, and its value is \( \chi(Y) \).

For integers \( p, t \geq 1 \), a \textbf{\( (p, t) \)-balloon} \( (P, Y) \) in \( G \) consists of an induced path \( P \) in \( G \) with vertices \( v_1\text{-}v_2\text{-}\cdots\text{-}v_p \) (in order), and a subset \( Y \subseteq V(G) \), satisfying the following properties:
\begin{itemize}[leftmargin=*, label=$\bullet$]
    \item \( v_1, \ldots, v_{p-1} \notin Y \), \( v_p \in Y \). None of \( v_1, \ldots, v_{p-2} \) have a neighbor in \( Y \), and if \( p \geq 2 \), then \( v_p \) is the unique neighbor of \( v_{p-1} \) in \( Y \).
    \item \( G[Y] \) is \( t \)-connected.
\end{itemize}
The value of the \( (p, t) \)-balloon \( (P, Y) \) is \( \chi(Z) \), where \( Z \) is the set of vertices in \( Y \) nonadjacent to \( v_p \) (hence \( v_p \in Z \)).
\vspace{2mm}

Nguyen et al.  proved the following lemma.

\begin{lemma} \emph{(\cite{Nguyen} Lemma 2.3)}\label{lem-mgun}
For every graph $G$, and all integers $p, q, s, t \geq 1$, if $G$ contains no $(p,t)$-balloon of value at least $q$, and $G$ contains no $t$-biclique of value $s$, then:
\[
\chi(G) \leq (\sum_{i=0}^{p-1} t^i)(s + t(2t + 9)) + t^p q
\]
\end{lemma}

\begin{theorem}\label{theo-scsys} \emph{(\cite{Scottsysp1} 3.2)}
Let $\eta, t \geq 1$ and $\zeta \geq 2$. For every rooted tree $(H,r)$ with height at most $\eta$ and spread at most $\zeta$, let $c = (\eta + 3)! |H|$. Then $\partial(G) \leq (|H| \zeta t)^c$ for every $H$-free graph $G$ that does not contain $K_{t,t}$ as a subgraph. 
\end{theorem}

\begin{theorem}\emph{(\cite{Randerath})}\label{theo-mgun}
Let \( G \) be a paw-free graph. If \( G \) is \( H \)-free, cross-free, or \( E \)-free, then:
\[
\chi(G) \leq \omega(G) + 1.
\]
\end{theorem}

\section{Results}

We define the following classes.\\

1) Let \(\mathcal{M}\) be the class of \emph{paw-free} graphs.
\vspace{2mm}  

A type of subdivision of the dart graph is considered, as illustrated in Figure \ref{fig1} (b).
\vspace{2mm}

 2) We define class~$\mathcal{L}$ to be the class of graphs that are both \emph{paw-free} and \emph{(sub-dart)-free}.
\\

\begin{figure}[H]
    \centering
    \begin{subfigure}[b]{0.3\textwidth}
        \centering
        \begin{tikzpicture}[scale=.6, transform shape]
            \node [draw, circle, fill=white] (a1) at (0,0) {};
            \node [draw, circle, fill=white] (a2) at (0,-1.5) {};
            \node [draw, circle, fill=white] (a3) at (0,-3) {};
            \node [draw, circle, fill=white] (a4) at (1.5,0) {};
            \node [draw, circle, fill=white] (a5) at (1.5,-1.5) {};
            \node [draw, circle, fill=white] (a6) at (1.5,-3) {};
            \node [scale=1.3] at (-0.5,0) {$a_1$};
            \node [scale=1.3] at (-0.5,-1.5) {$a_2$};
            \node [scale=1.3] at (-0.5,-3) {$a_3$};
            \node [scale=1.3] at (2,0) {$a_4$};
            \node [scale=1.3] at (2,-1.5) {$a_5$};
            \node [scale=1.3] at (2,-3) {$a_6$};
            \draw (a1) -- (a2) -- (a3);  
            \draw (a1) -- (a4);          
            \draw (a2) -- (a5);          
            \draw (a3) -- (a6);  
        \end{tikzpicture}
        \caption{$E-graph$}
        \label{fig:Egraph}
    \end{subfigure}
    \hspace{0.03\textwidth}
    \begin{subfigure}[b]{0.3\textwidth}
        \centering
        \begin{tikzpicture}[scale=0.6, transform shape]
            \node [draw, circle, fill=white] (l4) at (3,4) {};
            \node [draw, circle, fill=white] (m)  at (3,3) {};
            \node [draw, circle, fill=white] (l2) at (3,2) {};
            \node [draw, circle, fill=white] (l3) at (3,1) {};
            \node [draw, circle, fill=white] (l1) at (1,3) {};
            \node [draw, circle, fill=white] (l5) at (5,3) {};
            \draw (l4) -- (m) -- (l2) -- (l3);
            \draw (l1) -- (l2);
            \draw (l5) -- (l2);
            \draw (l1) -- (l4);
            \draw (l5) -- (l4);
        \end{tikzpicture}
        \caption{$sub-dart$}
    \end{subfigure}
    \hspace{0.03\textwidth}
    \begin{subfigure}[b]{0.3\textwidth}
        \centering
        \begin{tikzpicture}[scale=0.6, transform shape]
            \node[draw, circle, fill=white] (b1) at (0,0) {};
            \node[draw, circle, fill=white] (b2) at (-3,-2) {};
            \node[draw, circle, fill=white] (b3) at (0,-2) {};
            \node[draw, circle, fill=white] (b4) at (3,-2) {};
            \node[draw, circle, fill=white] (b5) at (-4,-4) {};
            \node[draw, circle, fill=white] (b6) at (-3,-4) {};
            \node[draw, circle, fill=white] (b7) at (-2,-4) {};
            \node[draw, circle, fill=white] (b8) at (-1,-4) {};
            \node[draw, circle, fill=white] (b9) at (0,-4) {};
            \node[draw, circle, fill=white] (b10) at (1,-4) {};
            \node[draw, circle, fill=white] (b11) at (2,-4) {};
            \node[draw, circle, fill=white] (b12) at (3,-4) {};
            \node[draw, circle, fill=white] (b13) at (4,-4) {};
            \node[anchor=south] at (b1.north) {$b_1$};
\node[anchor=east] at (b2.west) {$b_2$};
\node[anchor=south, xshift=7pt] at (b3.north) {$b_3$};
\node[anchor=west] at (b4.east) {$b_4$};
\node[anchor=north] at (b5.south) {$b_5$};
\node[anchor=north] at (b6.south) {$b_6$};
\node[anchor=north] at (b7.south) {$b_7$};
\node[anchor=north] at (b8.south) {$b_8$};
\node[anchor=north] at (b9.south) {$b_9$};
\node[anchor=north] at (b10.south) {$b_{10}$};
\node[anchor=north] at (b11.south) {$b_{11}$};
\node[anchor=north] at (b12.south) {$b_{12}$};
\node[anchor=north] at (b13.south) {$b_{13}$};
            \draw (b1) -- (b2);
            \draw (b1) -- (b3);
            \draw (b1) -- (b4);
            \draw (b2) -- (b5);
            \draw (b2) -- (b6);
            \draw (b2) -- (b7);
            \draw (b3) -- (b8);
            \draw (b3) -- (b9);
            \draw (b3) -- (b10);
            \draw (b4) -- (b11);
            \draw (b4) -- (b12);
            \draw (b4) -- (b13);
        \end{tikzpicture}
        \caption{$H(3,2)$}
    \end{subfigure}

    \vspace{1.2cm} 

    \begin{subfigure}[b]{0.3\textwidth}
        \centering
        \begin{tikzpicture}[scale=.6, transform shape]
            \node [draw, circle, fill=white] (a) at (0,0) {};
            \node [draw, circle, fill=white] (b) at (-1,-1.5) {};
            \node [draw, circle, fill=white] (c) at (1,-1.5) {};
            \node [draw, circle, fill=white] (d) at (0,1.2) {};
            \draw (a) -- (b);
            \draw (b) -- (c);
            \draw (c) -- (a);
            \draw (a) -- (d);
        \end{tikzpicture}
        \caption{$paw-graph$}
    \end{subfigure}
  
    \caption{}\label{fig1}
\end{figure}


\begin{definition} Let $T_{1}$ be a tree containing an induced path with $a_{1},...,a_{p}$ vertices, where the end vertex $a_p$ is adjacent to a vertex $x$ of a path $P_2=x,z$ and  a vertex $y$ of a path $P_1=y$. This tree is defined inductively for $p \geq 4$ \emph{(}Since problem has already been proven for radius two\emph{)}, \emph{(Figure \ref{figT1})}.
\end{definition}

\begin{definition} Let $T_{2}$ be a tree containing an induced path with $a_{1},...,a_{p}$ vertices, where the end vertex $a_p$ be adjacent to a vertex $x$ of a path $P_2=x,z$
and  a vertex $y$ of a path $P_2=y,t$. This tree is defined inductively for $p \geq 4$, \emph{(Figure \ref{figT2})}.
\end{definition}


\begin{figure}[H]
    \centering
    \begin{tabular}{cc} 

    \begin{minipage}{0.4\textwidth}
        \centering 
        \begin{tikzpicture}[scale=0.85, every node/.style={draw, circle, fill=white, inner sep=2.2pt}]
            \draw[thick] plot [smooth, tension=1] coordinates {(0,0) (1,0.3) (2,0.1) (3,0.2) (4,0)};
            \node[draw=none, fill=none] at (2,0.8) {p-path};
            \node (vp) at (4,0) {};
            \node (leaf) at (5,0.6) {};
            \draw (vp) -- (leaf);
            \node (a1) at (5,-0.5) {};
            \node (a2) at (6,-1) {};
            \draw (vp) -- (a1) -- (a2);
            \node at (vp) {};
            \node at (leaf) {};
            \node at (a1) {};
            \node at (a2) {};
        \end{tikzpicture}
        \caption{$T_1$}\label{figT1}
    \end{minipage}
    
    \begin{minipage}{0.4\textwidth}
        \centering 
        \begin{tikzpicture}[scale=0.85, every node/.style={draw, circle, fill=white, inner sep=2.2pt}]
            \draw[thick] plot [smooth, tension=1] coordinates {(0,0) (1,0.3) (2,0.1) (3,0.2) (4,0)};
            \node[draw=none, fill=none] at (2,0.8) {p-path};
            \node (vp) at (4,0) {};
            \node (b1) at (5,0.6) {};
            \node (b2) at (6,1.2) {};
            \draw (vp) -- (b1) -- (b2);
            \node (a1) at (5,-0.5) {};
            \node (a2) at (6,-1) {};
            \draw (vp) -- (a1) -- (a2);
            \node at (vp) {};
            \node at (a1) {};
            \node at (a2) {};
            \node at (b1) {};
            \node at (b2) {};
        \end{tikzpicture}
        \caption{$T_2$}\label{figT2}
    \end{minipage}

    \end{tabular}
\end{figure}

 \begin{flushleft} The contrapositive of the Theorem \ref{theo-mgun} shows that if $G\in \mathcal{M} $ and $\chi(G) > \omega(G) + 1$, then graph contains $E$-graph and $H$-graph and cross. 
\end{flushleft}

\begin{theorem}\label{theom-main1}
For all $d \geq 1$ and every $T_1$, there exists a polynomial $f$ such that for all $t \geq 1$, if a graph $G\in \mathcal{M} $ is $T_1$-free and does not contain $K_d(t)$ as a subgraph, then $\chi(G) \leq f(t)$.
\end{theorem}

\begin{proof}
We prove the theorem by induction on $d$.

If $d = 1$, then the result holds trivially since graphs without $K_1(t)$ have fewer than $t$ vertices.
Assume $d > 1$ and the result holds for $d-1$ and let $g(t)$ be the corresponding polynomial.

Define
\[
f(t) = (\sum_{i=0}^{p} t^i)\left(g(t) + 1 + t(2t + 9)\right) + t^{p+1} (dt+2).
\]

Let $G$ be a $T_1$-free graph that does not contain $K_{d}(t)$. Then we  show that $\chi(G) \leq f(t)$.
\vspace{2mm}

{\bf Claim 1.} $G$ contains no $(p+1, t)$-balloon of value $dt+2$.

Suppose $(P+1, Y)$ is such a $(p+1, t)$-balloon, then
\[
\chi(Y) \geq dt+2 \geq \omega(G)+2.
\]

By Theorem \ref{theo-mgun}, $Y$ contains an induced subgraph isomorphic to $E$-graph. Since $Y$ is connected, there exists a path between $v_{p}$ and at least one vertex of $E$. Let $Q$ be a minimal (and hence induced) path from $v_{p}$ to $E$.
If $v_{i}$, upon first reaching $E$, is adjacent to multiple vertices in $E$, then some of these adjacencies could potentially result in a paw. However, preventing such a configuration in this setting would violate the uniqueness of $v_i$ as the first point of contact with $E$, which is not acceptable in the given context.

Thus, the structure of the tree $T_1$ must always exist, regardless of whether $v_i$ is adjacent to only one vertex in $E$ or to multiple vertices in $E$, thereby establishing the desired contradiction {(Figure~\ref{figfind1})}.
\vspace{2mm}

{\bf Claim 2.} $G$ contains no $t$-biclique of value $g(t) + 1$. 

If $(X, Y)$ is such a $t$-biclique, from the inductive hypothesis $G[Y]$ contains $K_{d-1}(t)$, hence $G$ contains $K_{d}(t)$, a contradiction.

From Claims (1), (2) and Lemma \ref{lem-mgun}
\[
f(t) = (\sum_{i=0}^{p} t^i)\left(g(t) + 1 + t(2t + 9)\right) + t^{p+1} (dt+2).
\]

Thus the proof is completed.
\end{proof}

Theorem \ref{theodoe24}, by induction, yields the function \( f_1(t) \) as the desired polynomial bound.

\begin{align*}
w(t) &= s^4 t^s + s, \\
f_8(t) &= 3s d^{3s+2} w^{2s-1} t^{3s} \left( f_0(t) + {2s^{2d+2}}{t^{ds + s^2 + s}} \right), \\
f_5(t) = f_6(t) = f_7(t) &= 120 s d^{5s+1} w t^{5s} f_8(t), \\
f_3(t) = f_4(t) &= 2d^{s+1} w t^s f_5(t), \\
f_2(t) &= 2sdw f_3(t), \\
f_1(t) &= f_0 \left({s(s^2 + s + 1)t}^{120(s^2 + s + 1)}w \right) + 2t f_2(t).
\end{align*}\\

Therefore, the theorem also holds for \( H(3,2) \), that is, when \( s = 3 \).
\begin{theorem}\label{theom-main2}
For all $d \geq 1$ and every $T_2$, there exists a polynomial $f$ such that for all $t \geq 1$, if a graph $G\in \mathcal{L} $ is $T_2$-free and does not contain $K_d(t)$ as a subgraph, then $\chi(G) \leq f(t)$.
\end{theorem}
\begin{proof}
The proof follows by induction, similar to Theorem \ref{theom-main1}, with the difference that the $(p+1, t)$-balloon of value : $f_1(t)+1$ does not appear. Since \( H(3,2)\) appears within the balloon (for $s=3$), and similarly to the previous proof, There exists an induced path with \( p \) vertices leading to \( H(3,2) \).
Hence, the tree \( T_2 \) must occur in the graph.

Define
\[
f(t) = (\sum_{i=0}^{p} t^i)\left(g(t) + 1 + t(2t + 9)\right) + t^{p+1} (f_1(t)+1).
\]
Under the assumption that the graph is both \emph{paw-free} and \emph{(sub-dart)-free}, {(Figure~\ref{figfind2})} depicts the various configurations by which the path may connect to $H$. Configurations in which the path $Q$ is adjacent to only one vertex in $H(3,2)$ are not shown, as they are straightforward. Dashed edges indicate optional adjacencies and may or may not be present.
\end{proof}

\begin{theorem}\label{theom-main3}
For every integer \( d \geq 1 \) and every tree \( T \), there exists a polynomial \( f \) such that for all integers \( t \) with \( 1 \leq t \leq d - 1 \), if a graph \( G \) is \( T \)-free and contains neither \( K_d(t) \) nor \( K_t(t) \) as a subgraph, then \( \chi(G) \leq f(t) \).

\end{theorem}
 \begin{proof}
we prove the theorem by induction on $d$. If \( d = 1 \), then the result is obvious.
Assume $d > 1$ and the result holds for $d-1$. Let $g(t)$ be the corresponding polynomial.

Define
\[
f(t) = (\sum_{i=0}^{p-1} t^i)\left(g(t) + 1 + t(2t + 9)\right) + t^p (g(t)+1).
\]

Let $G$ be $T$-free graph that contains neither \( K_d(t) \) nor \( K_t(t) \) as a subgraph ; we will show that $\chi(G) \leq f(t)$.\\
\item[\bf Claim1:] $G$ contains no $(p, t)$-balloon of value : $g(t)+1$.
from the inductive hypothesis $G[Y]$ contains $K_{d-1}(t)$. Since  \( 1 \leq t \leq d - 1 \) and \( G \) does not contain $K_{t}(t)$ as a subgraph, it is a contradiction.\item[\bf Claim2:] $G$ contains no $t$-biclique of value $g(t) + 1$. A similar conclusion holds.\\
Thus the proof is completed.
\end{proof}

A graph $G$ is called \emph{$d$-degenerate} if every nonnull subgraph has a vertex of degree at most $d$. The degeneracy $\partial(G)$ is the smallest such $d$, Then $\chi(G) \leq \partial(G) + 1$ .\\
From Theorem \ref{theo-scsys}, we have:
\begin{observation}\label{obs-H(s,p)}
For every $H(s,p)$-free graph $G$ that does not contain $K_{t,t}$ as a subgraph, let $c = (p + 3)! |H|$. Then 
$\chi(G) \leq \beta_H(t)$. $(\beta_H(t)=(|H| s t)^c +1)$.
\end{observation}
The contrapositive of the Observation \ref{obs-H(s,p)} shows that if $t>\tau_2(G)$ and $\chi(G)>\beta_H(t)$ then the graph contains $H$.
For every tree $T$, there exist integers $s$ and $p$ such that $T \subseteq H(s,p)$. 
Without loss of generality, we consider the vertices of maximum degree in the tree $T$, and root the tree at each such vertex, one at a time. Let $p$ denote the minimum height among all the resulting rooted trees, and let $s$ be the maximum degree in $T$.
For integers $s$ and $p$, we define the function $\beta_T(t)$ as follows:
 $\beta_T(t)=((\sum_{i=0}^{p} s^{i+1})t)^{(p + 3)!\sum_{i=0}^{p} s^{i}} +1 $ .\\

For every tree T and $t,d \geq1$, the class \(\mathcal{F}_{T} = \mathcal{F}_1 \bigcup \mathcal{F}_{2}\) is 
defined as follows.

\begin{itemize}

    \item \(\mathcal{F}_1\): The class of graphs that \(\tau_2(G)< t\).
          
    \item  \(\mathcal{F}_{2}\): The class of graphs that,  \(\tau_2(G)\geq t\), and satisfy the condition: 
          If there exists a  $K_2(\beta_T(t))$ in the graph, then at least one of the parts of $K_2(\beta_T(t))$ contains $K_2(t)$. 
          It is evident that in this case, \(\tau_3(G)\geq t\).      
\end{itemize}
\begin{theorem}\label{theom-main4}
For all $d \geq 1$ and every tree $T$, there exists a polynomial $f$ such that for all $t \geq 1$, if a graph $G\in \mathcal{F}_{T}\ $ is $T$-free and does not contain $K_d(t)$ as a subgraph, then $\chi(G) \leq f(t)$.
\end{theorem}

\begin{proof}
For every tree $T$, there exist integers $s$ and $p$ such that $T \subseteq H(s,p)$. Clearly, if the graph contains \( H(s,p) \), then it must also contain \( T \).

We prove the theorem by induction on $d$. If \( d = 1 \), then the result is obvious.
Assume $d > 1$ and the result holds for $d-1$. Let $g(t)$ be the corresponding polynomial.

Define
\[
f(t) = (\sum_{i=0}^{p-1} t^i)\left(g(t) + 1 + t(2t + 9)\right) + t^p (g(\beta_T(t))+1).
\]

Let $G$ be $T$-free graph that does not contain $K_{d}(t)$; we will show that $\chi(G) \leq f(t)$.\\
a) If $G\in \mathcal {F}_{1}$, then by Observation \ref{obs-H(s,p)}, $\chi(G) \leq \beta_T(t)$. since we have assumed $f$ is a non-decreasing function; hence, $g$ is also non-decreasing so $\beta_T(t) \leq g(\beta_T(t))$.\\
b) If   $G\in \mathcal {F}_{2}$ The following two claims hold.

\begin{enumerate}
\item[\bf Claim1:] $G$ contains no $(p, t)$-balloon of value: $g(\beta_T(t))+1$.

Suppose $(P, Y)$ is such a $(p, t)$-balloon. Then,
\[
\chi(Y) \geq g(\beta_T(t))+1 
\]

from the inductive hypothesis $G[Y]$ contains $K_{d-1}(\beta_T(t))$.
if $G\in \mathcal{F}_{2}\ $,Since $G$ contains $ K_2(t)$ in this class,it follows that \( d \geq 3 \) hence $G$ contains $K_{d}(t)$, a contradiction.
\item[\bf Claim2:] $G$ contains no $t$-biclique of value $g(t) + 1$. A similar conclusion holds.
\end{enumerate}

From Claimes (1), (2) and Lemma 1,
\[
f(t) = (\sum_{i=0}^{p-1} t^i)\left(g(t) + 1 + t(2t + 9)\right) + t^p (g(\beta_T(t))+1).
\]

This proves Theorem 11.
\end{proof}

 \begin{flushleft}\textbf{Note:} The class $\mathcal{F}_2$ is not hereditary.
 \end{flushleft}

\begin{figure}[H]
    \centering
 \begin{tabular}{cccc} 
    \begin{minipage}{0.19\textwidth}
        \centering
        \begin{tikzpicture}[scale=.5, transform shape]
            \node [draw, circle, fill=white] (a1) at (0,0) {};
            \node [draw, circle, fill=white] (a2) at (0,-1.5) {a2};
            \node [draw, circle, fill=white] (a3) at (0,-3) {};
            \node [draw, circle, fill=white] (a4) at (1.5,0) {};
            \node [draw, circle, fill=white] (a5) at (1.5,-1.5) {};
            \node [draw, circle, fill=white] (a6) at (1.5,-3) {};
            \node [draw, circle, fill=white] (vi) at (-1.8,-1.5) {vi};
            \node [draw, circle, fill=white] (vi1) at (-3,-1.5) {};
            \draw[dotted, thick] (-4.2,-1.5) -- (-3.5,-1.5);
            \draw[line width=1pt, gray!60, preaction={draw=gray!30, line width=4pt}] (vi1) -- (vi);
            \draw[line width=1pt, gray!60, preaction={draw=gray!30, line width=4pt}] (vi) -- (a2);
            \draw[line width=1pt, gray!60, preaction={draw=gray!30, line width=4pt}] (vi) -- (a4);
            \draw[line width=1pt, gray!60, preaction={draw=gray!30, line width=4pt}] (a2) -- (a5);
            \draw (a1) -- (a2) -- (a3);
            \draw (a1) -- (a4);
            \draw (a3) -- (a6);
            \draw (vi) -- (a6);
        \end{tikzpicture}
    \end{minipage}
    
     \vspace{2mm}
    \begin{minipage}{0.19\textwidth}
        \centering
        \begin{tikzpicture}[scale=.5, transform shape]
            \node [draw, circle, fill=white] (a1) at (0,0) {};
            \node [draw, circle, fill=white] (a2) at (0,-1.5) {a2};
            \node [draw, circle, fill=white] (a3) at (0,-3) {};
            \node [draw, circle, fill=white] (a4) at (1.5,0) {};
            \node [draw, circle, fill=white] (a5) at (1.5,-1.5) {};
            \node [draw, circle, fill=white] (a6) at (1.5,-3) {};
            \node [draw, circle, fill=white] (vi) at (-1.8,-1.5) {};
            \node [draw, circle, fill=white] (vi1) at (-3,-1.5) {};
            \draw[dotted, thick] (-4.2,-1.5) -- (-3.5,-1.5);
            \draw[line width=1pt, gray!60, preaction={draw=gray!30, line width=4pt}] (vi1) -- (vi);
            \draw[line width=1pt, gray!60, preaction={draw=gray!30, line width=4pt}] (vi) -- (a2);
            \draw[line width=1pt, gray!60, preaction={draw=gray!30, line width=4pt}] (a2) -- (a5);
            \draw[line width=1pt, gray!60, preaction={draw=gray!30, line width=4pt}] (a2) -- (a3);
            \draw[line width=1pt, gray!60, preaction={draw=gray!30, line width=4pt}] (a3) -- (a6);
            \draw (a1) -- (a2);
            \draw (a1) -- (a4);
            \draw (vi) -- (a4);
        \end{tikzpicture}

    \end{minipage}\\

    \begin{minipage}{0.19\textwidth}
        \centering
        \begin{tikzpicture}[scale=.5, transform shape]
    \node [draw, circle, fill=white] (a1) at (0,0) {a1};
    \node [draw, circle, fill=white] (a2) at (0,-1.5) {};
    \node [draw, circle, fill=white] (a3) at (0,-3) {};
    \node [draw, circle, fill=white] (a4) at (1.5,0) {};
    \node [draw, circle, fill=white] (a5) at (1.5,-1.5) {};
    \node [draw, circle, fill=white] (a6) at (1.5,-3) {};
    \node [draw, circle, fill=white] (vi) at (-1.8,0) {vi};  
    \node [draw, circle, fill=white] (vi1) at (-3,0) {};   
    \draw[dotted, thick] (-4.2,0) -- (-3.5,0);
    \draw[line width=1pt, gray!60, preaction={draw=gray!30, line width=4pt}] (vi1) -- (vi);
    \draw[line width=1pt, gray!60, preaction={draw=gray!30, line width=4pt}] (vi) -- (a1);
    \draw[line width=1pt, gray!60, preaction={draw=gray!30, line width=4pt}] (vi) -- (a5);
    \draw[line width=1pt, gray!60, preaction={draw=gray!30, line width=4pt}] (a1) -- (a4);
    \draw (a1) -- (a2);
    \draw (a2) -- (a3);
    \draw (a2) -- (a5);
    \draw (a3) -- (a6);
\end{tikzpicture}

    \end{minipage}


    \begin{minipage}{0.19\textwidth}
        \centering
        \begin{tikzpicture}[scale=.5,transform shape]
    \node [draw, circle, fill=white] (a1) at (0,0) {a1};
    \node [draw, circle, fill=white] (a2) at (0,-1.5) {};
    \node [draw, circle, fill=white] (a3) at (0,-3) {};
    \node [draw, circle, fill=white] (a4) at (1.5,0) {};
    \node [draw, circle, fill=white] (a5) at (1.5,-1.5) {};
    \node [draw, circle, fill=white] (a6) at (1.5,-3) {};
    \node [draw, circle, fill=white] (vi) at (-1.8,0) {};
    \node [draw, circle, fill=white] (vi1) at (-3,0) {};
    \draw[dotted, thick] (-4.2,0) -- (-3.5,0);
    \draw[line width=1pt, gray!60, preaction={draw=gray!30, line width=4pt}] (vi1) -- (vi);
    \draw[line width=1pt, gray!60, preaction={draw=gray!30, line width=4pt}] (vi) -- (a1);
    \draw[line width=1pt, gray!60, preaction={draw=gray!30, line width=4pt}] 
        (vi) to[out=-20, in=210] (a6); 
    \draw[line width=1pt, gray!60, preaction={draw=gray!30, line width=4pt}] (a1) -- (a4);
    \draw (a1) -- (a2);
    \draw (a2) -- (a3);
    \draw (a2) -- (a5);
    \draw (a3) -- (a6);
\end{tikzpicture}
    \end{minipage}
    \begin{minipage}{0.19\textwidth}
            \centering
  \begin{tikzpicture}[scale=.5, transform shape]
    \node [draw, circle, fill=white] (a1) at (0,0) {a1};
    \node [draw, circle, fill=white] (a2) at (0,-1.5) {};
    \node [draw, circle, fill=white] (a3) at (0,-3) {};
    \node [draw, circle, fill=white] (a4) at (1.5,0) {};
    \node [draw, circle, fill=white] (a5) at (1.5,-1.5) {};
    \node [draw, circle, fill=white] (a6) at (1.5,-3) {};
    \node [draw, circle, fill=white] (vi) at (-1.8,0) {};
    \node [draw, circle, fill=white] (vi1) at (-3,0) {};
    \draw[dotted, thick] (-4.2,0) -- (-3.5,0);
    \draw[line width=1pt, gray!60, preaction={draw=gray!30, line width=4pt}] (vi1) -- (vi);
    \draw[line width=1pt, gray!60, preaction={draw=gray!30, line width=4pt}] (vi) -- (a1);
    \draw[line width=1pt, gray!60, preaction={draw=gray!30, line width=4pt}] (vi) -- (a3);
    \draw[line width=1pt, gray!60, preaction={draw=gray!30, line width=4pt}] (a1) -- (a4);
    \draw (a1) -- (a2);
    \draw (a2) -- (a3);
    \draw (a2) -- (a5);
    \draw (a3) -- (a6);
  \end{tikzpicture}
 
  \end{minipage}
\vspace{2mm}
  \begin{minipage}{0.19\textwidth}
  
\begin{tikzpicture}[scale=.5, transform shape]
    \node [draw, circle, fill=white] (a1) at (0,0) {a1};
    \node [draw, circle, fill=white] (a2) at (0,-1.5) {};
    \node [draw, circle, fill=white] (a3) at (0,-3) {};
    \node [draw, circle, fill=white] (a4) at (1.5,0) {};
    \node [draw, circle, fill=white] (a5) at (1.5,-1.5) {};
    \node [draw, circle, fill=white] (a6) at (1.5,-3) {};
    \node [draw, circle, fill=white] (vi) at (-1.8,0) {}; 
    \node [draw, circle, fill=white] (vi1) at (-3,0) {};   
    \draw[dotted, thick] (-4.2,0) -- (-3.5,0);
    \draw[line width=1pt, gray!60, preaction={draw=gray!30, line width=4pt}] (vi1) -- (vi);
    \draw[line width=1pt, gray!60, preaction={draw=gray!30, line width=4pt}] (vi) -- (a1);
    \draw[line width=1pt, gray!60, preaction={draw=gray!30, line width=4pt}] (vi) -- (a5);
    \draw[line width=1pt, gray!60, preaction={draw=gray!30, line width=4pt}] (a1) -- (a4);
    \draw (a1) -- (a2);
    \draw (a2) -- (a3);
    \draw (a2) -- (a5);
    \draw (a3) -- (a6);
   
   \draw (vi) to[out=-20, in=210] (a6);
\end{tikzpicture}

    \end{minipage}
    
    \begin{minipage}{0.19\textwidth}
        \centering
        \begin{tikzpicture}[scale=.5,transform shape]
    \node [draw, circle, fill=white] (a1) at (0,0) {a1};
    \node [draw, circle, fill=white] (a2) at (0,-1.5) {};
    \node [draw, circle, fill=white] (a3) at (0,-3) {};
    \node [draw, circle, fill=white] (a4) at (1.5,0) {};
    \node [draw, circle, fill=white] (a5) at (1.5,-1.5) {};
    \node [draw, circle, fill=white] (a6) at (1.5,-3) {};
    \node [draw, circle, fill=white] (vi) at (-1.8,0) {};
    \node [draw, circle, fill=white] (vi1) at (-3,0) {};
    \draw[dotted, thick] (-4.2,0) -- (-3.5,0);
    \draw[line width=1pt, gray!60, preaction={draw=gray!30, line width=4pt}] (vi1) -- (vi);
    \draw[line width=1pt, gray!60, preaction={draw=gray!30, line width=4pt}] (vi) -- (a1);
    \draw[line width=1pt, gray!60, preaction={draw=gray!30, line width=4pt}] (vi) -- (a3); 
    \draw[line width=1pt, gray!60, preaction={draw=gray!30, line width=4pt}] (a1) -- (a4);
    \draw (a1) -- (a2);
    \draw (a2) -- (a3);
    \draw (a2) -- (a5);
    \draw (a3) -- (a6);
    \draw (vi) -- (a5);
\end{tikzpicture}
    \end{minipage}\\
   \vspace{2mm}
    
    \begin{minipage}{0.19\textwidth}
        \centering
\begin{tikzpicture}[scale=0.5, transform shape]
    \node [draw, circle, fill=white] (a1) at (0,0) {};
    \node [draw, circle, fill=white] (a2) at (0,-1.5) {};
    \node [draw, circle, fill=white] (a3) at (0,-3) {};
    \node [draw, circle, fill=white] (a4) at (1.5,0) {a4};
    \node [draw, circle, fill=white] (a5) at (1.5,-1.5) {};
    \node [draw, circle, fill=white] (a6) at (1.5,-3) {};
    \node [draw, circle, fill=white] (vi) at (3.5,0) {vi};
    \node [draw, circle, fill=white] (vi1) at (5,0) {};
    \draw[dotted, thick] (vi1) -- (6,0);
    \draw (a1) -- (a2);  
    \draw (a1) -- (a4);
    \draw[line width=1pt, gray!60, preaction={draw=gray!30, line width=4pt}] (a2) -- (a3);    
    \draw[line width=1pt, gray!60, preaction={draw=gray!30, line width=4pt}] (a2) -- (a5);          
    \draw[line width=1pt, gray!60, preaction={draw=gray!30, line width=4pt}] (a3) -- (a6);  
    \draw[line width=1pt, gray!60, preaction={draw=gray!30, line width=4pt}] (vi) -- (a2);
    \draw (a4) -- (vi);
    \draw[line width=1pt, gray!60, preaction={draw=gray!30, line width=4pt}] (vi) -- (vi1);

\end{tikzpicture}
   \end{minipage} 
   
    \begin{minipage}{0.19\textwidth}
        \centering
\begin{tikzpicture}[scale=0.5, transform shape]

    \node [draw, circle, fill=white] (a1) at (0,0) {};
    \node [draw, circle, fill=white] (a2) at (0,-1.5) {};
    \node [draw, circle, fill=white] (a3) at (0,-3) {};
    \node [draw, circle, fill=white] (a4) at (1.5,0) {a4};
    \node [draw, circle, fill=white] (a5) at (1.5,-1.5) {};
    \node [draw, circle, fill=white] (a6) at (1.5,-3) {};
    \node [draw, circle, fill=white] (vi) at (3.5,0) {};
    \node [draw, circle, fill=white] (vi1) at (5,0) {};
    \draw[dotted, thick] (vi1) -- (6,0);
    \draw[line width=1pt, gray!60, preaction={draw=gray!30, line width=4pt}] (a1) -- (a2);  
    \draw (a1) -- (a4);
    \draw[line width=1pt, gray!60, preaction={draw=gray!30, line width=4pt}] (a2) -- (a3);    
    \draw[line width=1pt, gray!60, preaction={draw=gray!30, line width=4pt}] (a2) -- (a5);          
    \draw[line width=1pt, gray!60, preaction={draw=gray!30, line width=4pt}] (a3) -- (a6);  
    \draw[line width=1pt, gray!60, preaction={draw=gray!30, line width=4pt}] (vi) -- (a5);
    \draw (a4) -- (vi);
    \draw[line width=1pt, gray!60, preaction={draw=gray!30, line width=4pt}] (vi) -- (vi1);

\end{tikzpicture}
   \end{minipage} 
    \begin{minipage}{0.19\textwidth}
    \centering
\begin{tikzpicture}[scale=0.5, transform shape]

    \node [draw, circle, fill=white] (a1) at (0,0) {};
    \node [draw, circle, fill=white] (a2) at (0,-1.5) {};
    \node [draw, circle, fill=white] (a3) at (0,-3) {};
    \node [draw, circle, fill=white] (a4) at (1.5,0) {a4};
    \node [draw, circle, fill=white] (a5) at (1.5,-1.5) {};
    \node [draw, circle, fill=white] (a6) at (1.5,-3) {};
    \node [draw, circle, fill=white] (vi) at (3.5,0) {};
    \node [draw, circle, fill=white] (vi1) at (5,0) {};
    \draw[dotted, thick] (vi1) -- (6,0);
    \draw (a1) -- (a2);  
    \draw (a1) -- (a4);
    \draw[line width=1pt, gray!60, preaction={draw=gray!30, line width=4pt}] (a2) -- (a3);    
    \draw[line width=1pt, gray!60, preaction={draw=gray!30, line width=4pt}] (a2) -- (a5);          
    \draw[line width=1pt, gray!60, preaction={draw=gray!30, line width=4pt}] (a3) -- (a6);  
    \draw[line width=1pt, gray!60, preaction={draw=gray!30, line width=4pt}] (vi) -- (a2);
    \draw (a4) -- (vi);
    \draw[line width=1pt, gray!60, preaction={draw=gray!30, line width=4pt}] (vi) -- (vi1);
\end{tikzpicture}
\end{minipage}
    \begin{minipage}{0.19\textwidth}
    \centering
\begin{tikzpicture}[scale=0.5, transform shape]

    \node [draw, circle, fill=white] (a1) at (0,0) {};
    \node [draw, circle, fill=white] (a2) at (0,-1.5) {};
    \node [draw, circle, fill=white] (a3) at (0,-3) {};
    \node [draw, circle, fill=white] (a4) at (1.5,0) {a4};
    \node [draw, circle, fill=white] (a5) at (1.5,-1.5) {};
    \node [draw, circle, fill=white] (a6) at (1.5,-3) {};
    \node [draw, circle, fill=white] (vi) at (3.5,0) {};
    \node [draw, circle, fill=white] (vi1) at (5,0) {};
    \draw[dotted, thick] (vi1) -- (6,0);
    \draw[line width=1pt, gray!60, preaction={draw=gray!30, line width=4pt}] (a1) -- (a2);  
    \draw (a1) -- (a4);
    \draw[line width=1pt, gray!60, preaction={draw=gray!30, line width=4pt}] (a2) -- (a3);    
    \draw (a2) -- (a5);          
    \draw[line width=1pt, gray!60, preaction={draw=gray!30, line width=4pt}] (a3) -- (a6);  
    \draw[line width=1pt, gray!60, preaction={draw=gray!30, line width=4pt}] (vi) to[out=-50, in=40] (a3);
    \draw (a4) -- (vi);
    \draw[line width=1pt, gray!60, preaction={draw=gray!30, line width=4pt}] (vi) -- (vi1);

\end{tikzpicture}
\end{minipage}
   \begin{minipage}{0.19\textwidth}
    \centering
\begin{tikzpicture}[scale=0.5, transform shape]

    \node [draw, circle, fill=white] (a1) at (0,0) {};
    \node [draw, circle, fill=white] (a2) at (0,-1.5) {};
    \node [draw, circle, fill=white] (a3) at (0,-3) {};
    \node [draw, circle, fill=white] (a4) at (1.5,0) {a4};
    \node [draw, circle, fill=white] (a5) at (1.5,-1.5) {};
    \node [draw, circle, fill=white] (a6) at (1.5,-3) {};
    \node [draw, circle, fill=white] (vi) at (3.5,0) {};
    \node [draw, circle, fill=white] (vi1) at (5,0) {};
    \draw[dotted, thick] (vi1) -- (6,0);
    \draw (a1) -- (a2);  
    \draw (a1) -- (a4);
    \draw (a2) -- (a3);    
    \draw (a2) -- (a5);          
    \draw[line width=1pt, gray!60, preaction={draw=gray!30, line width=4pt}] (a3) -- (a6);  
    \draw[line width=1pt, gray!60, preaction={draw=gray!30, line width=4pt}] (vi) to[out=-50, in=40] (a6);
    \draw[line width=1pt, gray!60, preaction={draw=gray!30, line width=4pt}] (a4) -- (vi);
    \draw[line width=1pt, gray!60, preaction={draw=gray!30, line width=4pt}] (vi) -- (vi1);

\end{tikzpicture}
\end{minipage}\\
   \vspace{2mm}
   \begin{minipage}{0.19\textwidth}
    \centering
\begin{tikzpicture}[scale=0.5, transform shape]

    \node [draw, circle, fill=white] (a1) at (0,0) {};
    \node [draw, circle, fill=white] (a2) at (0,-1.5) {};
    \node [draw, circle, fill=white] (a3) at (0,-3) {};
    \node [draw, circle, fill=white] (a4) at (1.5,0) {a4};
    \node [draw, circle, fill=white] (a5) at (1.5,-1.5) {};
    \node [draw, circle, fill=white] (a6) at (1.5,-3) {};
    \node [draw, circle, fill=white] (vi) at (3.5,0) {};
    \node [draw, circle, fill=white] (vi1) at (5,0) {};
    \draw[dotted, thick] (vi1) -- (6,0);
    \draw (a1) -- (a2);  
    \draw[line width=1pt, gray!60, preaction={draw=gray!30, line width=4pt}] (a1) -- (a4);
    \draw (a2) -- (a3);    
    \draw (a2) -- (a5);          
    \draw (a3) -- (a6);  
    \draw (vi) -- (a2);
    \draw[line width=1pt, gray!60, preaction={draw=gray!30, line width=4pt}] (vi) to[out=-50, in=40] (a6);
    \draw[line width=1pt, gray!60, preaction={draw=gray!30, line width=4pt}] (a4) -- (vi);
    \draw[line width=1pt, gray!60, preaction={draw=gray!30, line width=4pt}] (vi) -- (vi1);

\end{tikzpicture}
\end{minipage}
  \begin{minipage}{0.19\textwidth}
    \centering
\begin{tikzpicture}[scale=0.5, transform shape]

    \node [draw, circle, fill=white] (a1) at (0,0) {};
    \node [draw, circle, fill=white] (a2) at (0,-1.5) {};
    \node [draw, circle, fill=white] (a3) at (0,-3) {};
    \node [draw, circle, fill=white] (a4) at (1.5,0) {a4};
    \node [draw, circle, fill=white] (a5) at (1.5,-1.5) {};
    \node [draw, circle, fill=white] (a6) at (1.5,-3) {};
    \node [draw, circle, fill=white] (vi) at (3.5,0) {};
    \node [draw, circle, fill=white] (vi1) at (5,0) {};
    \draw[dotted, thick] (vi1) -- (6,0);
    \draw (a1) -- (a2);  
    \draw (a1) -- (a4);
    \draw (a2) -- (a3);    
    \draw[line width=1pt, gray!60, preaction={draw=gray!30, line width=4pt}] (a2) -- (a5);          
    \draw (a3) -- (a6);  
   \draw (vi) -- (a4);
    \draw[line width=1pt, gray!60, preaction={draw=gray!30, line width=4pt}] (vi) to[out=-50, in=40] (a6);
    \draw[line width=1pt, gray!60, preaction={draw=gray!30, line width=4pt}] (a5) -- (vi);
    \draw[line width=1pt, gray!60, preaction={draw=gray!30, line width=4pt}] (vi) -- (vi1);

\end{tikzpicture}
\end{minipage}
  \begin{minipage}{0.19\textwidth}
    \centering
\begin{tikzpicture}[scale=0.5, transform shape]

    \node [draw, circle, fill=white] (a1) at (0,0) {};
    \node [draw, circle, fill=white] (a2) at (0,-1.5) {};
    \node [draw, circle, fill=white] (a3) at (0,-3) {};
    \node [draw, circle, fill=white] (a4) at (1.5,0) {a4};
    \node [draw, circle, fill=white] (a5) at (1.5,-1.5) {};
    \node [draw, circle, fill=white] (a6) at (1.5,-3) {};
    \node [draw, circle, fill=white] (vi) at (3.5,0) {};
    \node [draw, circle, fill=white] (vi1) at (5,0) {};
    \draw[dotted, thick] (vi1) -- (6,0);
    \draw (a1) -- (a2);  
    \draw[line width=1pt, gray!60, preaction={draw=gray!30, line width=4pt}] (a1) -- (a4);
    \draw (a2) -- (a3);    
    \draw (a2) -- (a5);          
    \draw (a3) -- (a6);  
    \draw (vi) -- (a5);
    \draw[line width=1pt, gray!60, preaction={draw=gray!30, line width=4pt}] (vi) to[out=-50, in=40] (a3);
    \draw[line width=1pt, gray!60, preaction={draw=gray!30, line width=4pt}] (a4) -- (vi);
    \draw[line width=1pt, gray!60, preaction={draw=gray!30, line width=4pt}] (vi) -- (vi1);
\end{tikzpicture}
\end{minipage}\\

  \vspace{2mm}
  
 \begin{minipage}{0.19\textwidth}
    \centering
\begin{tikzpicture}[scale=0.5, transform shape]

   \node [draw, circle, fill=white] (a1) at (0,0) {};
            \node [draw, circle, fill=white] (a2) at (0,-1.5) {};
            \node [draw, circle, fill=white] (a3) at (0,-3) {};
            \node [draw, circle, fill=white] (a4) at (1.5,0) {};
            \node [draw, circle, fill=white] (a5) at (1.5,-1.5) {a5};
            \node [draw, circle, fill=white] (a6) at (1.5,-3) {};
            \node [draw, circle, fill=white] (vi) at (3, -1.5) {vi};
            \node [draw, circle, fill=white] (vi1) at (4.5, -1.5) {};
            \draw (a1) -- (a2) -- (a3);  
            \draw (a1) -- (a4);          
            \draw[line width=1pt, gray!60, preaction={draw=gray!30, line width=4pt}] (a2) -- (a5);          
            \draw (a3) -- (a6);  
            \draw[line width=1pt, gray!60, preaction={draw=gray!30, line width=4pt}] (vi) -- (a4); 
            \draw[line width=1pt, gray!60, preaction={draw=gray!30, line width=4pt}] (a5) -- (vi);
            \draw[line width=1pt, gray!60, preaction={draw=gray!30, line width=4pt}] (vi) -- (vi1);
            \draw[dotted, thick] (vi1) -- +(1.2, 0);
            \draw[dashed, gray, thick] (vi) -- (a3); 
\end{tikzpicture}
\end{minipage}
   \begin{minipage}{0.19\textwidth}
    \centering
\begin{tikzpicture}[scale=0.5, transform shape]

   \node [draw, circle, fill=white] (a1) at (0,0) {};
            \node [draw, circle, fill=white] (a2) at (0,-1.5) {};
            \node [draw, circle, fill=white] (a3) at (0,-3) {};
            \node [draw, circle, fill=white] (a4) at (1.5,0) {};
            \node [draw, circle, fill=white] (a5) at (1.5,-1.5) {a5};
            \node [draw, circle, fill=white] (a6) at (1.5,-3) {};
            \node [draw, circle, fill=white] (vi) at (3, -1.5) {};
            \node [draw, circle, fill=white] (vi1) at (4.5, -1.5) {};
            \draw (a1) -- (a2) -- (a3);  
            \draw[line width=1pt, gray!60, preaction={draw=gray!30, line width=4pt}] (a1) -- (a4);          
            \draw (a2) -- (a5);          
            \draw (a3) -- (a6);  
            \draw[line width=1pt, gray!60, preaction={draw=gray!30, line width=4pt}] (a1) -- (vi);
            \draw[line width=1pt, gray!60, preaction={draw=gray!30, line width=4pt}] (a5) -- (vi);
            \draw[line width=1pt, gray!60, preaction={draw=gray!30, line width=4pt}] (vi) -- (vi1);
            \draw[dotted, thick] (vi1) -- +(1.2, 0); 
\end{tikzpicture}
\end{minipage}
  \begin{minipage}{0.19\textwidth}
    \centering
\begin{tikzpicture}[scale=0.5, transform shape]
   \node [draw, circle, fill=white] (a1) at (0,0) {};
            \node [draw, circle, fill=white] (a2) at (0,-1.5) {};
            \node [draw, circle, fill=white] (a3) at (0,-3) {};
            \node [draw, circle, fill=white] (a4) at (1.5,0) {};
            \node [draw, circle, fill=white] (a5) at (1.5,-1.5) {a5};
            \node [draw, circle, fill=white] (a6) at (1.5,-3) {};
            \node [draw, circle, fill=white] (vi) at (3, -1.5) {};
            \node [draw, circle, fill=white] (vi1) at (4.5, -1.5) {};
            \draw (a1) -- (a2) -- (a3);  
            \draw (a1) -- (a4);          
            \draw (a2) -- (a5);  
            \draw (vi) -- (a5);       
            \draw[line width=1pt, gray!60, preaction={draw=gray!30, line width=4pt}] (a3) -- (a6);  
            \draw[line width=1pt, gray!60, preaction={draw=gray!30, line width=4pt}] (a4) -- (vi);
            \draw[line width=1pt, gray!60, preaction={draw=gray!30, line width=4pt}] (a6) -- (vi);
            \draw[line width=1pt, gray!60, preaction={draw=gray!30, line width=4pt}] (vi) -- (vi1);
            \draw[dotted, thick] (vi1) -- +(1.2, 0); 
\end{tikzpicture}
\end{minipage}
 \begin{minipage}{0.19\textwidth}
    \centering
\begin{tikzpicture}[scale=0.5, transform shape]

   \node [draw, circle, fill=white] (a1) at (0,0) {};
            \node [draw, circle, fill=white] (a2) at (0,-1.5) {};
            \node [draw, circle, fill=white] (a3) at (0,-3) {};
            \node [draw, circle, fill=white] (a4) at (1.5,0) {};
            \node [draw, circle, fill=white] (a5) at (1.5,-1.5) {a5};
            \node [draw, circle, fill=white] (a6) at (1.5,-3) {};
            \node [draw, circle, fill=white] (vi) at (3, -1.5) {};
            \node [draw, circle, fill=white] (vi1) at (4.5, -1.5) {};
            \draw (a1) -- (a2) -- (a3);  
            \draw[line width=1pt, gray!60, preaction={draw=gray!30, line width=4pt}] (a1) -- (a4);          
            \draw (a2) -- (a5);          
            \draw (a3) -- (a6); 
            \draw (vi) -- (a5);  
            \draw[line width=1pt, gray!60, preaction={draw=gray!30, line width=4pt}] (a1) -- (vi);
            \draw[line width=1pt, gray!60, preaction={draw=gray!30, line width=4pt}] (a3) -- (vi);
            \draw[line width=1pt, gray!60, preaction={draw=gray!30, line width=4pt}] (vi) -- (vi1);
            \draw[dotted, thick] (vi1) -- +(1.2, 0); 
\end{tikzpicture}
\end{minipage}
  
\end{tabular}

\caption{Finding $T_1$} 
\label{figfind1}
        \vspace{-0.10em}
\end{figure}
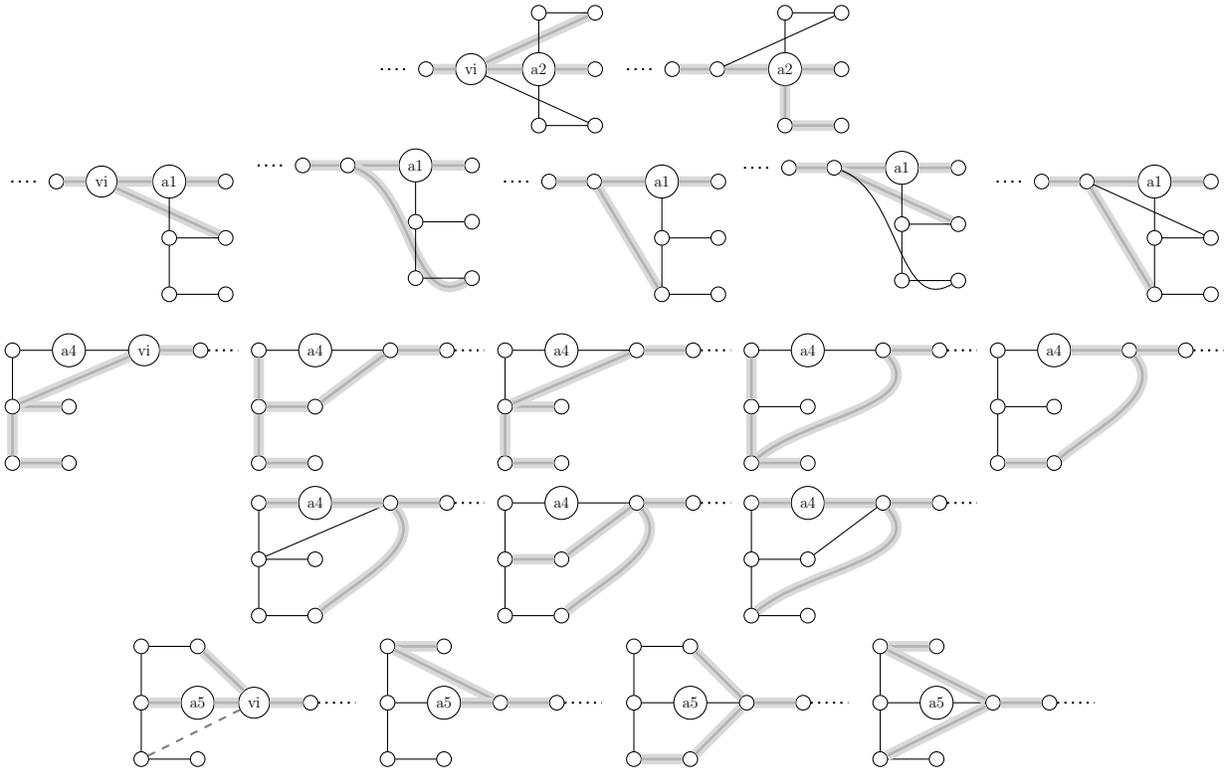

\begin{figure}[H]

    \centering
 \begin{tabular}{ccc} 

\begin{minipage}{0.19\textwidth}
\begin{tikzpicture}[scale=0.5, transform shape]
\node[draw, circle, fill=white] (vi) at (0,2) {vi};
\node[draw, circle, fill=white] (vi1) at (0,3.5) {};
\draw (vi1) -- +(0,1.2); 
\draw[dotted] (vi1) -- +(0,1.2);
\draw[line width=1pt, gray!60, preaction={draw=gray!30, line width=4pt}] (vi1) -- (vi);
\node[draw, circle, fill=white] (b1) at (0,0) {b1};
\draw[line width=1pt, gray!60, preaction={draw=gray!80, line width=4pt}] (vi) -- (b1);
\node[draw, circle, fill=white] (b2) at (-3,-2) {};
\node[draw, circle, fill=white] (b3) at (0,-2) {};
\node[draw, circle, fill=white] (b4) at (3,-2) {};
\node[draw, circle, fill=white] (b5) at (-4,-4) {};
\node[draw, circle, fill=white] (b6) at (-3,-4) {};
\node[draw, circle, fill=white] (b7) at (-2,-4) {};
\node[draw, circle, fill=white] (b8) at (-1,-4) {};
\node[draw, circle, fill=white] (b9) at (0,-4) {};
\node[draw, circle, fill=white] (b10) at (1,-4) {};
\node[draw, circle, fill=white] (b11) at (2,-4) {};
\node[draw, circle, fill=white] (b12) at (3,-4) {};
\node[draw, circle, fill=white] (b13) at (4,-4) {};
\draw (b1) -- (b2);
\draw[line width=1pt, gray!60, preaction={draw=gray!30, line width=4pt}] (b1) -- (b3);
\draw[line width=1pt, gray!60, preaction={draw=gray!30, line width=4pt}] (b1) -- (b4);
\draw (b2) -- (b5);
\draw (b2) -- (b6);
\draw (b2) -- (b7);
\draw (b3) -- (b8);
\draw[line width=1pt, gray!60, preaction={draw=gray!30, line width=4pt}] (b3) -- (b9);
\draw (b3) -- (b10);
\draw[line width=1pt, gray!60, preaction={draw=gray!30, line width=4pt}] (b4) -- (b11);
\draw (b4) -- (b12);
\draw (b4) -- (b13);
\draw[dashed, gray, thick] (vi) to[out=20, in=90] (b13);
\draw[dashed, gray, thick] (vi) to[out=10, in=40] (b8);
\draw[dashed, gray, thick] (vi) to[out=155, in=90] (b5);

\end{tikzpicture}
\end{minipage}

 \hspace{2cm} 
  
\begin{minipage}{0.19\textwidth}
  \centering
\begin{tikzpicture}[scale=0.5, transform shape]
\node[draw, circle, fill=white] (vi1) at (0,3.5) {};
\node[draw, circle, fill=white] (vi) at (0,2) {vi};
\node[draw, circle, fill=white] (b1) at (0,0) {};
\node[draw, circle, fill=white] (b2) at (-3,-2) {b2};
\node[draw, circle, fill=white] (b3) at (0,-2) {};
\node[draw, circle, fill=white] (b4) at (3,-2) {};
\node[draw, circle, fill=white] (b5) at (-4,-4) {};
\node[draw, circle, fill=white] (b6) at (-3,-4) {};
\node[draw, circle, fill=white] (b7) at (-2,-4) {};
\node[draw, circle, fill=white] (b8) at (-1,-4) {};
\node[draw, circle, fill=white] (b9) at (0,-4) {};
\node[draw, circle, fill=white] (b10) at (1,-4) {};
\node[draw, circle, fill=white] (b11) at (2,-4) {};
\node[draw, circle, fill=white] (b12) at (3,-4) {};
\node[draw, circle, fill=white] (b13) at (4,-4) {};
\draw[line width=1pt, gray!60, preaction={draw=gray!30, line width=4pt}] (vi1) -- (vi);
\draw (vi1) -- +(0,1.2); 
\draw[dotted] (vi1) -- +(0,1.2);
\draw[line width=1pt, gray!60, preaction={draw=gray!80, line width=4pt}] (vi) -- (b2);
\draw[line width=1pt, gray!60, preaction={draw=gray!30, line width=4pt}] (vi) -- (b4);
\draw (b1) -- (b2);
\draw (b1) -- (b3);
\draw (b1) -- (b4);
\draw[dashed, gray, thick] (vi) to[out=10, in=40] (b3);
\draw[line width=1pt, gray!60, preaction={draw=gray!30, line width=4pt}] (b2) -- (b5);
\draw (b2) -- (b6);
\draw (b2) -- (b7);
\draw (b3) -- (b8);
\draw (b3) -- (b9);
\draw (b3) -- (b10);
\draw[line width=1pt, gray!60, preaction={draw=gray!30, line width=4pt}] (b4) -- (b11);
\draw (b4) -- (b12);
\draw (b4) -- (b13);

\end{tikzpicture}
\end{minipage}
\hspace{2cm} 

\begin{minipage}{0.19\textwidth}
  \centering
\begin{tikzpicture}[scale=0.5, transform shape]
\node[draw, circle, fill=white] (vi1) at (0,3.5) {};
\node[draw, circle, fill=white] (vi) at (0,2) {vi};
\node[draw, circle, fill=white] (b1) at (0,0) {};
\node[draw, circle, fill=white] (b2) at (-3,-2) {b2};
\node[draw, circle, fill=white] (b3) at (0,-2) {};
\node[draw, circle, fill=white] (b4) at (3,-2) {};
\node[draw, circle, fill=white] (b5) at (-4,-4) {};
\node[draw, circle, fill=white] (b6) at (-3,-4) {};
\node[draw, circle, fill=white] (b7) at (-2,-4) {};
\node[draw, circle, fill=white] (b8) at (-1,-4) {};
\node[draw, circle, fill=white] (b9) at (0,-4) {};
\node[draw, circle, fill=white] (b10) at (1,-4) {};
\node[draw, circle, fill=white] (b11) at (2,-4) {};
\node[draw, circle, fill=white] (b12) at (3,-4) {};
\node[draw, circle, fill=white] (b13) at (4,-4) {};
\draw[line width=1pt, gray!60, preaction={draw=gray!30, line width=4pt}] (vi1) -- (vi);
\draw (vi1) -- +(0,1.2); 
\draw[dotted] (vi1) -- +(0,1.2);
\draw[line width=1pt, gray!60, preaction={draw=gray!80, line width=4pt}] (vi) -- (b2);
\draw[line width=1pt, gray!60, preaction={draw=gray!30, line width=4pt}] (vi) -- (b4);
\draw (b1) -- (b2);
\draw (b1) -- (b3);
\draw (b1) -- (b4);
\draw[dashed, gray, thick] (vi) to[out=12, in=40] (b8);
\draw[dashed, gray, thick] (vi) to[out=15, in=40] (b9);
\draw[dashed, gray, thick] (vi) to[out=17, in=40] (b10);
\draw[line width=1pt, gray!60, preaction={draw=gray!30, line width=4pt}] (b2) -- (b5);
\draw (b2) -- (b6);
\draw (b2) -- (b7);
\draw (b3) -- (b8);
\draw (b3) -- (b9);
\draw (b3) -- (b10);
\draw[line width=1pt, gray!60, preaction={draw=gray!30, line width=4pt}] (b4) -- (b11);
\draw (b4) -- (b12);
\draw (b4) -- (b13);

\end{tikzpicture}
\end{minipage}\\


\begin{minipage}{0.19\textwidth}
\vspace{1cm}
  \centering
\begin{tikzpicture}[scale=0.5, transform shape]
\node[draw, circle, fill=white] (vi1) at (0,3.5) {};
\node[draw, circle, fill=white] (vi) at (0,2) {vi};
\node[draw, circle, fill=white] (b1) at (0,0) {};
\node[draw, circle, fill=white] (b2) at (-3,-2) {b2};
\node[draw, circle, fill=white] (b3) at (0,-2) {};
\node[draw, circle, fill=white] (b4) at (3,-2) {};
\node[draw, circle, fill=white] (b5) at (-4,-4) {};
\node[draw, circle, fill=white] (b6) at (-3,-4) {};
\node[draw, circle, fill=white] (b7) at (-2,-4) {};
\node[draw, circle, fill=white] (b8) at (-1,-4) {};
\node[draw, circle, fill=white] (b9) at (0,-4) {};
\node[draw, circle, fill=white] (b10) at (1,-4) {};
\node[draw, circle, fill=white] (b11) at (2,-4) {};
\node[draw, circle, fill=white] (b12) at (3,-4) {};
\node[draw, circle, fill=white] (b13) at (4,-4) {};
\draw[line width=1pt, gray!60, preaction={draw=gray!30, line width=4pt}] (vi1) -- (vi);
\draw (vi1) -- +(0,1.2); 
\draw[dotted] (vi1) -- +(0,1.2);
\draw[line width=1pt, gray!60, preaction={draw=gray!80, line width=4pt}] (vi) -- (b2);
\draw[line width=1pt, gray!60, preaction={draw=gray!30, line width=4pt}] (vi) -- (b8);
\draw (b1) -- (b2);
\draw (b1) -- (b3);
\draw (b1) -- (b4);
\draw[dashed, gray, thick] (vi) to[out=12, in=24] (b11);
\draw[dashed, gray, thick] (vi) to[out=15, in=40] (b12);
\draw[dashed, gray, thick] (vi) to[out=17, in=40] (b13);
\draw[dashed, gray, thick] (vi) to[out=17, in=40] (b10);
\draw[dashed, gray, thick] (vi) to[out=17, in=40] (b9);  
\draw[line width=1pt, gray!60, preaction={draw=gray!30, line width=4pt}] (b2) -- (b5);
\draw (b2) -- (b6);
\draw (b2) -- (b7);
\draw[line width=1pt, gray!60, preaction={draw=gray!30, line width=4pt}] (b3) -- (b8);
\draw (b3) -- (b9);
\draw (b3) -- (b10);
\draw (b4) -- (b11);
\draw (b4) -- (b12);
\draw (b4) -- (b13);

\end{tikzpicture}
\end{minipage}

\hspace{2cm}
\begin{minipage}{0.19\textwidth}
\vspace{1cm}
  \centering
\begin{tikzpicture}[scale=0.5, transform shape]
\node[draw, circle, fill=white] (vi1) at (0,3.5) {};
\node[draw, circle, fill=white] (vi) at (0,2) {vi};
\node[draw, circle, fill=white] (b1) at (0,0) {};
\node[draw, circle, fill=white] (b2) at (-3,-2) {};
\node[draw, circle, fill=white] (b3) at (0,-2) {};
\node[draw, circle, fill=white] (b4) at (3,-2) {};
\node[draw, circle, fill=white] (b5) at (-4,-4) {b5};
\node[draw, circle, fill=white] (b6) at (-3,-4) {};
\node[draw, circle, fill=white] (b7) at (-2,-4) {};
\node[draw, circle, fill=white] (b8) at (-1,-4) {};
\node[draw, circle, fill=white] (b9) at (0,-4) {};
\node[draw, circle, fill=white] (b10) at (1,-4) {};
\node[draw, circle, fill=white] (b11) at (2,-4) {};
\node[draw, circle, fill=white] (b12) at (3,-4) {};
\node[draw, circle, fill=white] (b13) at (4,-4) {};
\draw[line width=1pt, gray!60, preaction={draw=gray!30, line width=4pt}] (vi1) -- (vi);
\draw (vi1) -- +(0,1.2);
\draw[dotted] (vi1) -- +(0,1.2);
\draw[line width=1pt, gray!60, preaction={draw=gray!80, line width=4pt}] (vi) to[out=180, in=90] (b5);
\draw[line width=1pt, gray!60, preaction={draw=gray!30, line width=4pt}] (vi) -- (b4);
\draw (b1) -- (b2);
\draw (b1) -- (b3);
\draw (b1) -- (b4);
\draw[dashed, gray, thick] (vi) to[out=10, in=30] (b3);
\draw[dashed, gray, thick] (vi) to[out=-10, in=50] (b6);
\draw[dashed, gray, thick] (vi) to[out=-5, in=50] (b7);
\draw[line width=1pt, gray!60, preaction={draw=gray!30, line width=4pt}] (b2) -- (b5);
\draw (b2) -- (b6);
\draw (b2) -- (b7);
\draw (b3) -- (b8);
\draw (b3) -- (b9);
\draw (b3) -- (b10);
\draw[line width=1pt, gray!60, preaction={draw=gray!30, line width=4pt}] (b4) -- (b11);
\draw (b4) -- (b12);
\draw (b4) -- (b13);
\end{tikzpicture}
\end{minipage}
\hspace{2cm}
\begin{minipage}{0.19\textwidth}
\vspace{1cm}
  \centering
\begin{tikzpicture}[scale=0.5, transform shape]
\node[draw, circle, fill=white] (vi1) at (0,3.5) {};
\node[draw, circle, fill=white] (vi) at (0,2) {vi};
\node[draw, circle, fill=white] (b1) at (0,0) {};
\node[draw, circle, fill=white] (b2) at (-3,-2) {};
\node[draw, circle, fill=white] (b3) at (0,-2) {};
\node[draw, circle, fill=white] (b4) at (3,-2) {};
\node[draw, circle, fill=white] (b5) at (-4,-4) {b5};
\node[draw, circle, fill=white] (b6) at (-3,-4) {};
\node[draw, circle, fill=white] (b7) at (-2,-4) {};
\node[draw, circle, fill=white] (b8) at (-1,-4) {};
\node[draw, circle, fill=white] (b9) at (0,-4) {};
\node[draw, circle, fill=white] (b10) at (1,-4) {};
\node[draw, circle, fill=white] (b11) at (2,-4) {};
\node[draw, circle, fill=white] (b12) at (3,-4) {};
\node[draw, circle, fill=white] (b13) at (4,-4) {};
\draw[line width=1pt, gray!60, preaction={draw=gray!30, line width=4pt}] (vi1) -- (vi);
\draw (vi1) -- +(0,1.2); 
\draw[dotted] (vi1) -- +(0,1.2);
\draw[line width=1pt, gray!60, preaction={draw=gray!80, line width=4pt}] (vi) to[out=180, in=90] (b5);
\draw[line width=1pt, gray!60, preaction={draw=gray!30, line width=4pt}] (vi) -- (b4);
\draw (b1) -- (b2);
\draw (b1) -- (b3);
\draw (b1) -- (b4);
\draw[dashed, gray, thick] (vi) to[out=10, in=30] (b8);
\draw[dashed, gray, thick] (vi) to[out=0, in=50] (b6);
\draw[dashed, gray, thick] (vi) to[out=0, in=50] (b7);
\draw[dashed, gray, thick] (vi) to[out=10, in=30] (b9);
\draw[dashed, gray, thick] (vi) to[out=10, in=30] (b10);
\draw[line width=1pt, gray!60, preaction={draw=gray!30, line width=4pt}] (b2) -- (b5);
\draw (b2) -- (b6);
\draw (b2) -- (b7);
\draw (b3) -- (b8);
\draw (b3) -- (b9);
\draw (b3) -- (b10);
\draw[line width=1pt, gray!60, preaction={draw=gray!30, line width=4pt}] (b4) -- (b11);
\draw (b4) -- (b12);
\draw (b4) -- (b13);

\end{tikzpicture}
\end{minipage}\\

\begin{minipage}{0.19\textwidth}
\vspace{1cm}
  \centering
\begin{tikzpicture}[scale=0.5, transform shape]
\node[draw, circle, fill=white] (vi1) at (0,3.5) {};
\node[draw, circle, fill=white] (vi) at (0,2) {vi};
\node[draw, circle, fill=white] (b1) at (0,0) {};
\node[draw, circle, fill=white] (b2) at (-3,-2) {};
\node[draw, circle, fill=white] (b3) at (0,-2) {};
\node[draw, circle, fill=white] (b4) at (3,-2) {};
\node[draw, circle, fill=white] (b5) at (-4,-4) {b5};
\node[draw, circle, fill=white] (b6) at (-3,-4) {};
\node[draw, circle, fill=white] (b7) at (-2,-4) {};
\node[draw, circle, fill=white] (b8) at (-1,-4) {};
\node[draw, circle, fill=white] (b9) at (0,-4) {};
\node[draw, circle, fill=white] (b10) at (1,-4) {};
\node[draw, circle, fill=white] (b11) at (2,-4) {};
\node[draw, circle, fill=white] (b12) at (3,-4) {};
\node[draw, circle, fill=white] (b13) at (4,-4) {};
\draw[line width=1pt, gray!60, preaction={draw=gray!30, line width=4pt}] (vi1) -- (vi);
\draw (vi1) -- +(0,1.2); 
\draw[dotted] (vi1) -- +(0,1.2);
\draw[line width=1pt, gray!60, preaction={draw=gray!80, line width=4pt}] (vi) to[out=180, in=90] (b5);
\draw (b1) -- (b2);
\draw (b1) -- (b3);
\draw (b1) -- (b4);
\draw[dashed, gray, thick] (vi) to[out=0, in=50] (b6);
\draw[dashed, gray, thick] (vi) to[out=0, in=50] (b7);
\draw[dashed, gray, thick] (vi) to[out=0, in=50] (b8);
\draw[dashed, gray, thick] (vi) to[out=0, in=50] (b9);
\draw[dashed, gray, thick] (vi) to[out=0, in=50] (b10);
\draw[dashed, gray, thick] (vi) to[out=0, in=50] (b11);
\draw[dashed, gray, thick] (vi) to[out=0, in=50] (b12);
\draw[line width=1pt, gray!60, preaction={draw=gray!30, line width=4pt}] (vi) to[out=10, in=50] (b13);
\draw[line width=1pt, gray!60, preaction={draw=gray!30, line width=4pt}] (b2) -- (b5);
\draw (b2) -- (b6);
\draw (b2) -- (b7);
\draw (b3) -- (b8);
\draw (b3) -- (b9);
\draw (b3) -- (b10);
\draw (b4) -- (b11);
\draw (b4) -- (b12);
\draw[line width=1pt, gray!60, preaction={draw=gray!30, line width=4pt}] (b4) -- (b13);

\end{tikzpicture}
\end{minipage}
\end{tabular}

\caption{Finding $T_2$} \label{figfind2}
        \vspace{0.0em}
\end{figure}
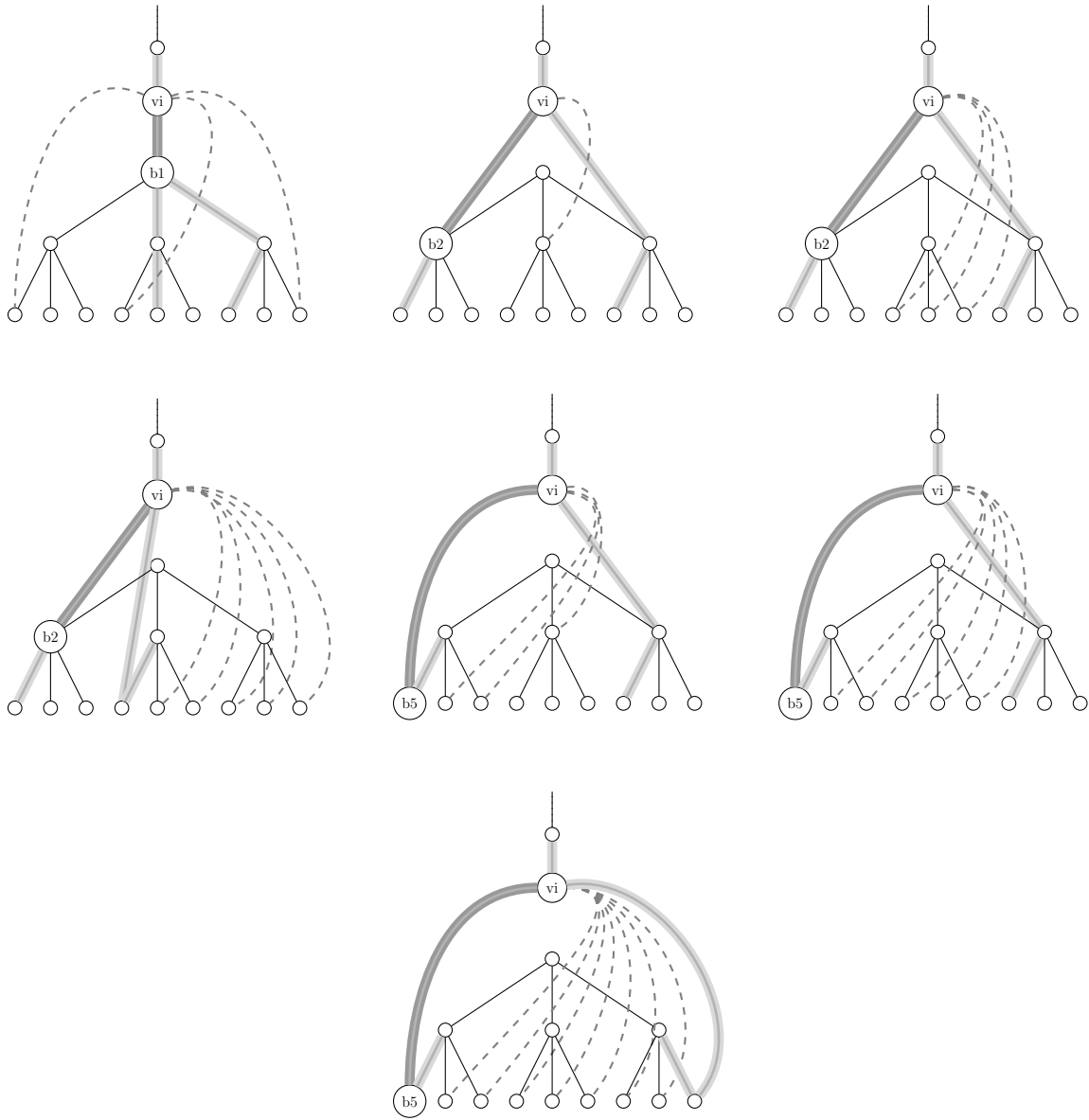

\end{document}